\documentclass{article}
\usepackage{amsmath,amssymb,amsfonts,amsthm,graphicx,subfigure,moreverb,latexsym,algorithm}
\usepackage{amsmath,amssymb,amsthm,amsfonts,mathrsfs,latexsym,filecontents,xcolor,multicol} 
\usepackage[mathcal]{euscript}
\usepackage[all]{xy}
\usepackage[below]{placeins} 
\usepackage[english]{babel}
\usepackage[utf8x]{inputenc}
\newtheorem{thm}{Theorem}[section]
\newtheorem{lem}[thm]{Lemma}
\newtheorem{exa}[thm]{Example}
\newtheorem{pro}[thm]{Proposition}
\newtheorem{defn}[thm]{Definition}
\newtheorem{cor}[thm]{Corollary}
\newtheorem{rem}[thm]{Remark}

\newcommand {\Z}{\mathbb Z} 
\newcommand {\F}{\mathbb F} 
\newcommand {\N}{\mathbb N} 
\title{A search for c-Wieferich primes}
\author{Alex Samuel Bamunoba\footnote{Department of Mathematics, Makerere University, Uganda, (\textit{bamunoba@aims.ac.za} or \textit{bamunoba@cns.mak.ac.ug})}\;\footnote{``This work was in part carried out at the Department of Mathematics of Stockholm University with the financial support from the Makerere-Sida Bilateral Programme Phase IV, Project 316 ``Capacity building in Mathematics and its applications".}\\ Jonas Bergstr\"om \footnote{Department of Mathematics, Stockholm University, Sweden, (\textit{jonasb@math.su.se})}}
\date{}
\begin{document}

\maketitle

\begin{abstract}
Let $q$ be a power of a prime number $p$, $\F_q$ be a finite field with $q$ elements and $\mathcal{G}$ be a subgroup of $(\F_q,+)$ of order $p$. We give an existence criterion and an algorithm for computing maximally $\mathcal{G}$-fixed c-Wieferich primes in $\F_q[T]$. Using the criterion, we study how c-Wieferich primes behave in $\F_q[T]$ extensions.
\end{abstract}


\section{Introduction}

Let $p$ be a prime number, $q=p^{m}$ for some $m\in\Z_{\geqslant 1}$, $\F_q$ be a finite field of $q$ elements and $T$ be an indeterminate. In addition, let $\mathcal{A}:=\F_q[T]$ be the ring of polynomials in $T$ defined over $\F_q$ and $\mathcal{A}_+$ be the set of monic polynomials in $\mathcal{A}$. We shall use the symbols $\mathcal{P}$ and $\mathcal{Q}$ to denote monic irreducible (also called prime) polynomials in $\mathcal{A}$. The symbol $\mathcal{Q}_d$ will denote any prime polynomial in $\mathcal{A}$ of degree $d$. For each $i\in \Z_{\geqslant 1}$, we let $[i]:=T^{q^i}-T$ and $L_i:=[i][i-1]\cdots [1]$. The symbol $[i]$ represents the product of all prime polynomials in $\mathcal{A}$ of degree dividing $i$ while $L_i$ represents the least common multiple of monic polynomials in $\mathcal{A}$ of degree $i$. We shall adopt the convention that, ``the value of an empty product is $1$" and therefore, set $[0]:=1=:L_0$.

Let $\mathbf{k}:=\F_q(T)$ be the field of fractions of $\mathcal{A}$, $\mathcal{K}$ be a fixed algebraic closure of $\mathbf{k}$ and $\tau:\mathcal{K}\to \mathcal{K}, \alpha\mapsto \alpha^q$ be the Frobenius $\F_q$-automorphism of $\mathcal{K}$. Let $\mathbf{k}\{\tau\}$ be the polynomial ring in $\tau$ over $\mathbf{k}$ with multiplication satisfying the commutation relation $\tau \theta =\theta^q\tau$ for any $\theta\in \mathbf{k}$. It is clear that $\mathbf{k}\{\tau\}$ is an $\F_q$-algebra, (sometimes called, the twisted polynomial ring over $\mathbf{k}$). Let $\mathscr{A}(\mathbf{k})$ be the set of additive polynomials over $\mathbf{k}$. Recall that, a polynomial $g\in \mathbf{k}[X]$ is said to be additive over $\mathbf{k}$ if $g(X+Y)=g(X)+g(Y)$ as a polynomial in $\mathbf{k}[X,Y]$. With the usual addition of polynomials in $\mathscr{A}(\mathbf{k})$ and multiplication given by composition of polynomial mappings, $\mathscr{A}(\mathbf{k})$ becomes a ring, called the ring of additive polynomials over $\mathbf{k}$. To each $G\in \mathbf{k}\{\tau\}$, we associate an element $g\in \mathscr{A}(\mathbf{k})$ defined by letting $G$ act on the indeterminate $X$ via $\tau$. For example, if $G=\tau+T\tau^0$, then $g=(\tau+T\tau^0)X=X^q+TX$. This gives a bijection between $\mathbf{k}\{\tau\}$ and $\mathscr{A}(\mathbf{k})$.

The ring homomorphism $\rho:\mathcal{A}\to\mathbf{k}\{\tau\}$ characterised by $T\mapsto \tau+T\tau^0$ is called the Carlitz $\mathcal{A}$-module homomorphism. This ring homomorphism is in fact a homomorphism of $\F_q$-algebras. The mapping $\rho$ is an $\mathcal{A}$-module homomorphism because it induces a ``new $\mathcal{A}$-action" on $\mathcal{K}$ as follows: consider the mapping $\ast:\mathcal{A}\times \mathcal{K}\to \mathcal{K}$ defined by $(N,\alpha)\mapsto\rho^N\alpha$, where $\rho^N$ is the Carlitz $N$-endomorphism over $\mathcal{A}$. Since to each $\rho^N\in \mathbf{k}\{\tau\}$, there is a corresponding additive polynomial $\rho_N(X)$, called the Carlitz $N$-polynomial (actually $\rho^N$ is defined over $\mathcal{A}$, since the image of $\mathcal{A}$ under $\rho$ is a subring of $\mathcal{A}\{\tau\}$). The action $\ast$ is equivalently defined as: $N\ast\alpha=\rho^N\alpha=\rho_N(\alpha)$, evaluating the polynomial $\rho_N(X)$ at $\alpha$. We call the pair $(\rho,\mathcal{K})=:\mathscr{C}(\mathcal{K})$, the Carlitz $\mathcal{A}$-module (or Carlitzification of $\mathcal{K}$). The Carlitz $\mathcal{A}$-module homomorphism $\rho$ is the simplest example of a sign normalised rank one Drinfeld $\mathcal{A}$-module. Technically, $\rho$ can also be thought of as the functor from the category of $\mathcal{A}$-algebras to the category of $\mathcal{A}$-modules which sends an $\mathcal{A}$-algebra $\mathcal{K}$ to the unique $\mathcal{A}$-module $\mathscr{C}(\mathcal{K})$ which has $\mathcal{K}$ as the underlying Abelian group, and such that the (left) multiplication by $T$ of an $\alpha\in \mathcal{K}$ is $\rho_{T}(\alpha)=\alpha^q+T\alpha$. For an introduction to the general theory of Drinfeld $\mathcal{A}$-modules, see \cite{DGOSS, MROSEN} and \cite{DTHAKUR}.

With this background, we are in a position to define the notion of c-Wieferich (or Carlitz-Wieferich) primes. 
\begin{defn}\label{cwprimes}
Let $a\in \mathcal{A}-\{0\}$. A prime polynomial $\mathcal{P}$ in $\mathcal{A}$ is said to be a c-Wieferich prime to base $a$ if 
\begin{align}\label{acwprime}
\rho_{\mathcal{P}}(a)\equiv a^{q^{\deg(\mathcal{P})}}(\bmod\;\mathcal{P}^2).
\end{align} 
\end{defn}
This notion of c-Wieferich primes was introduced in 1994 by D. Thakur \cite{dinesh1994iwasawa}. In \cite{DSTHAKURR}, D. Thakur showed that, if $a$ is a $p$-th power, (i.e., the derivative of $a$ is $0$), then \eqref{acwprime} is equivalent to $\rho_{\mathcal{P}}(a)\equiv a(\bmod\;\mathcal{P}^2)$. Since for any $N$ in $\mathcal{A}$, the Carlitz $N$-polynomial $\rho_N(X)$ is $\F_q$-linear, it follows that the congruence $\rho_{\mathcal{P}}(a)\equiv a(\bmod\;\mathcal{P}^2)$ for any $a\in \F_q^{\times}$ is equivalent to $\rho_{\mathcal{P}-1}(1)\equiv 0(\bmod\;\mathcal{P}^2)$. Therefore, a prime polynomial $\mathcal{P}$ in $\mathcal{A}$ is a c-Wieferich prime to base $\alpha\in \F_q^{\times}$ (or base $1$) if $\rho_{\mathcal{P}-1}(1)\equiv 0(\bmod\;\mathcal{P}^2)$. In this paper, the term ``c-Wieferich prime" will refer to a c-Wieferich prime to base $1$. For example, $T^6 + T^4 + T^3 + T^2 + 2T + 2$, $T^9 + T^6 + T^4 + T^2 + 2T + 2, T^{12} + 2T^{10} + T^9 + 2T^4 + 2T^3 + T^2 + 1$ and $T^{15} + T^{13} + T^{12} + T^{11} + 2T^{10} + 2T^7 + 2T^5 + 2T^4 + T^3 + T^2 + T + 1$ are all c-Wieferich primes in $\F_3[T]$. These are the only c-Wieferich primes in $\F_3[T]$ of degree at most $60$. However, in the case of $q=2$, there is an anomolous behaviour where all the prime polynomials in $\F_2[T]-\{T, T+1\}$ are c-Wieferich primes. This is because $\rho_{\mathcal{P}-1}(1)=0$ for any prime $\mathcal{P}\in \F_2[T]$ of degree at least $2$. 

Examples of c-Wieferich primes in $\F_q[T]$ (for $q\geqslant 3$) are quite rare and it is computationally intensive to use Definition \ref{cwprimes} to find them. For this reason, we search for other conditions that can be used in order to gain more insight and be able to compute these objects. To achieve this, we need some extra notation. For each $n\in\Z_{\geqslant 0}$, let $F_n$ denote the numerator of $\sum_{j=0}^{n}(-1)^j(L_j)^{-1}$ written as a rational polynomial without common factors. It is an easy exercise to show that the polynomials $F_n$ are monic in $\mathcal{A}$ and solutions to the recurrence relation $F_0=1$, $F_{i}=(-1)^{i}+[i]F_{i-1}$, for $i=1,2,\ldots$. For the proof, see \cite[Lemma 4.2]{ASB2}. 
\begin{pro}[{\cite[Proposition 4.3]{ASB2}}]\label{ASB2}
$\mathcal{P}$ is a c-Wieferich prime in $\mathcal{A}$ if and only if $F_{\deg(\mathcal{P})-1}\equiv 0(\bmod\;\mathcal{P})$. 
\end{pro}
Proposition \ref{ASB2} was independently discovered by D. Thakur in his work on c-Wieferich primes, see \cite{DSTHAKURR}. The last three examples of c-Wieferich primes in $\F_3[T]$ listed on page 2 were obtained using Proposition \ref{ASB2}. 

There is a link between c-Wieferich primes and $\zeta_{\mathcal{A}}(1)$, the Carlitz-Goss zeta value at $1$. To see this, let 
\begin{align*}
\zeta_{\mathcal{A}}(1):=\sum_{a\in \mathcal{A}_+}\frac{1}{a},~~\text{ and }~~\zeta_{\mathcal{A},\mathcal{P}}(1):=\sum_{\substack{a\in \mathcal{A}_+\\ (a,\mathcal{P})=1}}\frac{1}{a},
\end{align*}
where $\zeta_{\mathcal{A},\mathcal{P}}(1)$ is the Carlitz-Goss $\mathcal{P}$-adic zeta value at $1$. For $q>2$, D. Thakur showed that, a prime $\mathcal{P}\in \mathcal{A}$ is a c-Wieferich prime if and only if $\mathcal{P}$ divides $\zeta_{\mathcal{A},\mathcal{P}}(1)$ if and only if $\mathcal{P}^2$ divides $\zeta_{\mathcal{A}}(1)$, see \cite[Theorem 5]{DSTHAKURR}. For the link between c-Wieferich primes and Mersenne primes, refer to \cite{dongquan}.

Proposition \ref{ASB2} leads us to a naive algorithm for computing c-Wieferich primes, see  Algorithm \ref{ALG01}.
\begin{algorithm}[ht]
\caption{\; Computing c-Wieferich primes I.}
\label{ALG01}
\begin{enumerate}
\item[] \textbf{Input}: $q$ - size of the base field of $\mathcal{A}$, and $n$ - the degree of c-Wieferich primes required.
\item[] \textbf{Output}: Product of c-Wieferich primes of degree less than or equal to $n$, (in fact dividing $n$).

1. $F\longleftarrow 1$, $B$ an empty list.

2. for $i=1$ to $n-1$

\hspace{1cm} $F\longleftarrow (-1)^i+[i]F$, (where $[i]=T^{q^i}-T$)

3. $B\longleftarrow {\rm GCD}([n],F)$

\item[] \textbf{Return}: $B$
\end{enumerate}
\end{algorithm}

In Algorithm \ref{ALG01}, one recursively computes $F_i$ for $i=1$ to $n-1$ and lastly, the gcd of $F_{n-1}$ and $[n]$. This yields a product of c-Wieferich primes in $\mathcal{A}$ of degree dividing $n$. Another way to quickly check for existence of c-Wieferich primes of degree say dividing $n$ is to compute ${\rm resultant}([n],F_{n-1})$, the resultant of $[n]$ and $F_{n-1}$, if ${\rm resultant}([n],F_{n-1})\neq 0$, then there are some c-Wieferich primes of degree dividing $n$. In Table \ref{eval_table1} we give experimental evidence for the existence of c-Wieferich primes in $\mathcal{A}$, obtained using Algorithm \ref{ALG01} implemented in SAGE Mathematics Software. We used $\F_q:=\F_p(t)$ with $f_{\rm min}^{t}(X)$, the minimum polynomial of $t$ over $\F_p$.

\begin{table}[ht!]
\centering \footnotesize
\begin{tabular}{cccc}
\hline
\hline
$\F_q$ & $f_{\rm min}^t(X)$ & c-Wieferich primes in $\F_q[T]$ of least degrees \\
\hline
\hline
$\F_{3}$  &  & $T^6 + T^4 + T^3 + T^2 + 2T + 2$, $T^9 + T^6 + T^4 + T^2 + 2T + 2$ \\
$\F_{2^2}$  & $X^2+X+1$ & $T^2+T+t$, $T^2+T+t+1$ \\
$\F_{5}$  &  & $T^5+T+1$, $T^{10}+3T^6+4T^5 +T^2+T+1$ \\
$\F_{7}$  & & $T^7 + 6T + 3$, $T^{14} + 5T^8 + 5T^7 + T^2 + 2T + 3$ \\
$\F_{2^3}$  & $X^3 + X + 1$ & $T^2 + T + 1$, $T^2 + T + t + 1$, $T^2 + T + t^2 + 1$, $T^2 + T + t^2 + t + 1$, $T^4 + T + 1$, $T^4 + T + t + 1$,\\ && $T^4 + T + t^2 + 1$, $T^4 + T + t^2 + t + 1$\\
$\F_{3^2}$ & $X^2 + 2X + 2$ & $T^{3} + \left(t + 1\right) T + 1$, $T^{3} + \left(t + 1\right) T + t$, $T^{3} + \left(t + 1\right) T + 2 t + 2$, $T^{3} + \left(2 t + 2\right) T + 1$,\\ & & $T^{3} + \left(2 t + 2\right) T + t + 1$, $T^{3} + \left(2 t + 2\right) T + 2 t + 1$ \\ 
$\F_{2^4}$ & $X^4+X+1$ & $T^2+T+t^3$, \ldots, $T^2+T+t^3+t^2+t+1$, {\color{blue}{$T^3+T+1$}}, ${\color{red}{T^3+T^2+1}}$,\ldots\\ & & ${\color{blue}{T^3+(t^3+t^2+t+1)T^2+(t^3+t+1)T+t}}$, \ldots\\
$\F_{3^3}$ & $X^3 + 2X + 1$ & --\\
$\F_{5^2}$ & $X^2+4X+2$ & $T^5 + 4T + 3, T^5 + 4T + t, T^5 + 4T + 2t + 2, T^5 + 4T + 3t + 4$ and $T^5 + 4T + 4t + 1$\\
$\F_{7^2}$ & $X^2+6X+3$ & $T^7 + 6T + t + 1, T^7 + 6T + 2t + 4, T^7 + 6T + 3t, T^7 + 6T + 4t + 3, $\\ & & $T^7 + 6T + 6t + 2, T^7 + 6T + 5, T^7 + 6T + 5t + 6$ \\
\hline
\end{tabular}
\caption{Examples of c-Wieferich primes in some $\F_q[T]$ with $q$ small.}
\label{eval_table1}
\end{table}

The examples in Table \ref{eval_table1} indicate the non-existence of c-Wieferich primes of degree $1$. This is in accordance with \cite[Corollary 4.5]{ASB2} which tells us that, there are no c-Wieferich primes of degree $1$. Although Algorithm \ref{ALG01} works well for small degrees, its misgiving lies in the exponential growth of the degrees of the polynomials involved at the third step of Algorithm \ref{ALG01}. This motivates us to search for more efficient algorithms.

D. Thakur \cite{DSTHAKURR} points out that, the naive guess that the degree of the c-Wieferich prime  is a multiple of the characteristic of $\F_q$ fails in characteristic 2 by V. Mauduit's examples in \cite{veronique}. We will give more counter-examples in the case when $q$ is even. Despite this, we still believe that Thakur's naive guess is true in odd characteristic and it is one of the motivations for this paper. The remainder of the paper is organized as follows. In Section \ref{A}, we will give preliminary results from the theory of finite fields that will be used in proving the criterion for maximally $\mathcal{G}$-fixed c-Wieferich primes. In Section \ref{C}, we establish two criteria for maximally $\mathcal{G}$-fixed c-Wieferich primes in $\mathcal{A}$. In Section \ref{D}, we explain two algorithms (based on Theorems \ref{IKL} and \ref{IKL2}) and use them to demonstrate computations of $\mathcal{G}$-fixed c-Wieferich primes. In Section \ref{E}, we give a mixed bag of results regarding the distribution of c-Wieferich primes in constant field extensions of $\F_q[T]$.

\section{Some results from finite field theory} \label{A}

\begin{thm}\label{CII?}
Let $p$ be the characteristic of $\F_q$, $\beta\in \F_q^{\times}$ and $\gamma\in \F_{q^r}^{\times}$ for some $r\in \Z_{\geqslant 1}$. The trinomial $T^p-\beta T-\gamma$ is irreducible over $\F_{q^r}$ if and only if it has no root in $\F_{q^r}$.
\end{thm}
\begin{proof}
Modify the arguments in the proof of \cite[Theorem 3.78]{RLidl01}.
\end{proof}
The prime polynomials of the form $T^p-\beta T-\gamma$ where $\beta\in \F_q^{\times}$ and $\gamma\in \F_{q^r}$, $r\in\Z_{\geqslant 1}$ will be called \textit{almost} Artin-Schreier primes for $\F_{q^r}[T]$. If $q=p$, and $\beta=1$, then these are the Artin-Schreier primes for $\F_p[T]$. The next result describes precisely the decomposition of polynomials of the form $T^q-T-\alpha\in \F_{q^r}[T]$.

\begin{thm}\label{BIII?}
Let $q\geqslant 3$, $\alpha\in\F_{q^r}$ for some $r\in\Z_{\geqslant 1}$. If ${\rm Tr}_{\F_{q^r}/\F_q}(\alpha)\neq 0$, then $T^q-T-\alpha$ factors as 
\begin{align}\label{IUO}
T^q-T-\alpha =(T^p-\beta T-\gamma_1)(T^p-\beta T-\gamma_2)\cdots (T^p-\beta T-\gamma_{p^{-1}q}), \text{ where }\beta=({\rm Tr}_{\F_{q^r}/\F_q}(\alpha))^{p-1}
\end{align}
and $\gamma_1,\ldots,\gamma_{p^{-1}q}\in\F_{q^r}^{\times}$ are distinct. If ${\rm Tr}_{\F_{q^r}/\F_q}(\alpha)=0$, then $T^q-T-\alpha$ splits completely in $\F_{q^r}[T]$.
\end{thm} 
\begin{proof}
Modify the arguments in the proof of \cite[Theorem 3.80]{RLidl01}.
\end{proof}

\begin{exa}\label{firstexample}
Let $\F_{3^4}:=\F_3(t)$ where $f^t_{\rm min}(X)=X^4 + 2X^3 + 2$ is the minimal polynomial of $t$ over $\F_3$. Then $t\in \F_{3^4}$ is a primitive element and ${\rm Tr}_{\F_{3^4}/\F_{3^2}}(t)=t^9+t=t^3 + t^2\neq 0$. By Theorem \ref{BIII?}, there exists a unique $\beta\in \F_{9}^{\times}$ such that $T^9-T-t$ factors in $\F_{3^4}[T]$ into almost Artin-Schreier primes. Using SAGE,
\begin{align*}
T^9-T-t =&(T^{3} + (2 t^{3} + 2 t^{2} + 2)T + 2 t^{2} + t)(T^{3} + (2 t^{3} + 2 t^{2} + 2)T + t^{3} + t)\\&(T^{3} +(2 t^{3} + 2 t^{2} + 2)T + 2 t^{3} + t^{2} + t).
\end{align*}
Taking another primitive element as $\alpha=t^3 + t + 1\in \F_{3^4}$. We find that ${\rm Tr}_{\F_{3^4}/\F_{3^2}}(t^3+t+1)=0$. By Theorem \ref{BIII?}, the polynomial $T^9-T-(t^3+t+1)$ splits completely in $\F_{3^4}[T]$. Using SAGE, we find that
\begin{align*}
T^9-T-(t^3+t+1)=&(T + t^{2} + 2 t)(T + t^{2} + 2 t + 1)(T + t^{2} + 2 t + 2)\\&(T + t^{3} + 2 t^{2} + 2 t) (T + t^{3} + 2 t^{2} + 2 t + 1)(T + t^{3} + 2 t^{2} + 2 t + 2)\\&(T + 2 t^{3} + 2 t)(T + 2 t^{3} + 2 t + 1)(T + 2 t^{3} + 2 t + 2).
\end{align*}
\end{exa}

There are a number of ways that $\F_q$ acts on $\F_q[T]$ as a subgroup of ${\rm GL}_2(\F_q)$. Of all these, we prefer the translation action $\star$ because it preserves monicity of polynomials and fixes $F_i$'s, (hence the property of being a c-Wieferich prime). For each $\alpha\in \F_q$, we set $\alpha\star f=f(T+\alpha)$ and the stabiliser of $f$ to be the set ${\rm stab}(f):=\{\alpha\in \F_q: \alpha\star f=f\}$. Let $\mathcal{G}$ be a subgroup of $\F_q$ and $f\in \F_q[T]$; we say $f$ is a $\mathcal{G}$-fixed (or a $\mathcal{G}$-translation invariant) polynomial if $\mathcal{G}\subseteq {\rm stab}(f)$. Note that, every $f\in \F_q[T]$ is $\{0\}$-fixed. We also say that, a polynomial $f$ is a maximally $\mathcal{G}$-fixed polynomial if ${\rm stab}(f)=\mathcal{G}$, i.e., amongst the subgroups of $\F_q$ that fix $f$, $\mathcal{G}$ is the ``largest". If a polynomial $f\in \F_q[T]$ is maximally $\{0\}$-fixed, then (with abuse of language), we also call it a non-fixed polynomial in $\F_q[T]$. The set of non-fixed polynomials is non-empty, since $T$ is always a member and the set of maximally $\F_q$-fixed polynomials is also non-empty, since the constant polynomial $1$ is always a member of this set. Moreover, one can easily prove that these sets are infinite.

For each subgroup $\mathcal{G}$ of $\F_q$, we define $[1]_{\mathcal{G}}:=\prod_{\alpha\in \mathcal{G}}(T+\alpha)=\prod_{\alpha\in \mathcal{G}}(T-\alpha)$. It is clear that $[1]_{\mathcal{G}}$ divides $[1]$, where $[1]=T^q-T$. In addition, if $\#\mathcal{G}=p$, then there exists a $\chi\in \F_q^{\times}$ such that $\mathcal{G}=\chi \F_p$ and that $[1]_{\mathcal{G}}=T^p-\beta T$ where $\beta=\chi^{p-1}$. The almost Artin-Schreier prime polynomials in $\F_{q^r}[T]$ (i.e., primes of the form $T^p-\beta T-\gamma$ where $\beta\in \F_q^{\times}$ and $\gamma\in \F_{q^r}$) are by construction $\mathcal{G}$-fixed with $\mathcal{G}=w\F_p$ where $w$ is a solution to the equation $X^{p-1}-\beta=0$. Therefore, any prime polynomial $\mathcal{P}$ in $\mathcal{A}$ that decomposes as a product of such primes with fixed $\beta=\chi^{p-1}$ is a $\mathcal{G}$-fixed prime. This characterization will be used in Theorem \ref{IKL}.

The motivation for studying $\mathcal{G}$-fixed primes in $\F_q[T]$ arose from the following example, (see \cite[Page 356]{ASB2}): consider the field $\F_{3^2}:=\F_3(t)$ where $t$ is such that $t^2+2t+2=0$. Using SAGE, we obtained $T^3+(t+1)T+1$ as one of the c-Wieferich primes in $\F_{3^2}[T]$. Let $\mathcal{G}_1=\{0,1,2\}$ and $\mathcal{G}_2=\{0,t+2,2t+1\}$, (both additive subgroups of the additive group $\F_{3^2}$ each of order $3$). It is easy to check that $T^3+(t+1)T+1$ is invariant under translation by elements of $\mathcal{G}_2$ but not $\mathcal{G}_1$. In this case, we say that $\mathcal{P}$ is a $\mathcal{G}_2$-fixed c-Wieferich prime. 

\begin{pro}
Let $f$ be a monic polynomial in $\F_q[T]$ and $\mathcal{G}$ be a subgroup of $\F_q$. Then $f$ is a $\mathcal{G}$-fixed polynomial if and only if there is some $g\in \F_q[T]$ such that $f=g([1]_{\mathcal{G}})$.
\end{pro}

\begin{proof}
$(\Leftarrow)$ This follows immediately from the fact that $[1]_{\mathcal{G}}$ is $\mathcal{G}$-fixed. $(\Rightarrow)$ For the converse, assume that $f$ is a monic $\mathcal{G}$-fixed polynomial of degree $n$. This  implies that $g_1=f-f(0)$ is also a $\mathcal{G}$-fixed monic polynomial. Moreover, $T$ divides $g_1$ since $g_1(0)=0$. Since $g_1$ is a $\mathcal{G}$-fixed monic polynomial, $g_1$ is divisible by all the translates $T+\alpha$, $\alpha\in \mathcal{G}$ and so, $[1]_{\mathcal{G}}$ divides $g_1$. Put $f_1=[1]_{\mathcal{G}}^{-s}g_1$, where $s$ is the number of times $T$ divides $g_1$, then repeat the procedure. Since $f$ is a polynomial, this procedure terminates (after at most $n$ steps). Looking at the sequence of operations in reverse reveals that $f$ is a monic polynomial in $[1]_{\mathcal{G}}$.
\end{proof}
From \cite[Theorem 7]{REIS20181087}, it follows that there are no primes $\mathcal{P}$ in $\mathcal{A}$ whose stabiliser has order $>p$, the characteristic of $\F_q$. So if $f$ is a prime in $\mathcal{A}$, then $f$ is fixed by a subgroup of $\F_q$ of order at most $p$. In the subsequent discussions, the phrase ``$\mathcal{G}$-fixed prime polynomials" implies that $\mathcal{G}$ is a subgroup of $\F_q$ of order at most $p$. 

\begin{thm}\label{LOKI}
Let $\mathcal{G}$ be a subgroup of the additive group of $\F_q$ of order $p$ and $f$ be a $\mathcal{G}$-fixed prime polynomial in $\mathcal{A}$ of degree $ps$. Then $f$ factors into a product of almost Artin-Schreier $\mathcal{G}$-fixed primes in $\F_{q^s}[T]$.
\end{thm}

\begin{proof}
From \cite[Theorem 2.7(b)]{REIS20181087}, it follows that the degree of $f$ is divisible by $p$, say equal to $ps$. Since $f$ is irreducible over $\F_q$ with $\zeta$ as one of its roots, it follows that $\F_{q^{ps}}=\F_q(\zeta)$ is the splitting field of $f$ over $\F_q$. Since $f$ is a $\mathcal{G}$-fixed polynomial in $\mathcal{A}$, for each root $\zeta\in\F_{q^{ps}}$ of $f$, the element $\zeta+a$ is also a root of $f$ for any $a\in \mathcal{G}$. Define $g:=\prod_{a\in \mathcal{G}}(T-(\zeta+a))$. Since $\mathcal{G}$ is a subgroup of $\F_q$ of order $p$, it follows that $\mathcal{G}=\chi \F_p$ for some $\chi\in \F_q^{\times}$. Since $f$ is irreducible over $\F_q$, and $\zeta+\chi$ is one of its roots, it follows that there exists a $t\in \Z_{\geqslant 1}$ such that $\zeta^{q^t}=\zeta+\chi$ and so $\zeta^{q^{pt}}=(\zeta+\chi)^{q^{(p-1)t}}=\zeta^{q^{(p-1)t}}+\chi=\zeta^{q^{(p-2)t}}+2\chi=\cdots=\zeta$. Galois theory then tells us that $g\in \F_{q^s}[T]$ and that $g=\prod_{a\in \mathcal{G}}(T-(\zeta+a))=(T-\zeta)\prod_{i\in\F_p^{\times}}(T-\zeta-i\chi)=T^p-\chi^{p-1}T-(\zeta^p-\chi^{p-1}\zeta)$, a $\mathcal{G}$-fixed polynomial. Since $g$ has no root in $\F_{q^s}$, Theorem \ref{CII?} tells us that $g$ is irreducible over $\F_{q^s}$ of almost Artin-Schreier type in $\F_{q^s}[T]$. So $f$ decomposes into almost Artin-Schreier $\mathcal{G}$-fixed primes in $\F_{q^s}[T]$.
\end{proof}

\section{Criterion for c-Wieferich primes in $\F_q[T]$}\label{C}
In this section, we let $\mathcal{G}$ be an additive subgroup of $\F_q$ of order at-most $p$ and derive two criteria used to compute $\mathcal{G}$-fixed c-Wieferich primes in $\mathcal{A}$. Before, we do this, let us first state and prove Lemma \ref{ME} which will be useful at a few places in the deduction of the main results of this section, i.e., Theorems \ref{IKL} and \ref{IKL2}.
\begin{lem}\label{ME}
Let $s,\ell\in \Z_{\geqslant 0}$ and $\mathcal{Q}_d$ be a degree $d$ prime in $\mathcal{A}$ and  $0\leqslant \ell <d$. If $F_i$ is the polynomial defined recursively as $F_0=1$ and $F_{i}=(-1)^i+[i]F_{i-1}$, for $i\in\Z_{\geqslant 1}$, then $F_{sd+\ell}\equiv (-1)^{sd}F_{\ell}(\bmod\;\mathcal{Q}_d)$.
\end{lem}
\begin{proof}
We have that $[sd+\ell]=[sd]^{q^{\ell}}+[\ell]\equiv [\ell](\bmod\;\mathcal{Q}_d)$, where $\mathcal{Q}_d$ is a degree $d$ prime polynomial in $\mathcal{A}$. We proceed by induction on $\ell$ with the base case $F_{sd}=(-1)^{sd}+[sd]F_{sd-1}\equiv (-1)^{sd}F_0(\bmod\;\mathcal{Q}_d)$. Then for $\ell>0$, by induction $F_{sd+\ell}=(-1)^{sd+\ell}+[sd+\ell]F_{sd+\ell-1}\equiv (-1)^{sd+\ell}+[\ell](-1)^{sd}F_{\ell-1}=(-1)^{sd}F_{\ell}(\bmod\;\mathcal{Q}_d)$.
\end{proof}  
\begin{defn}
Let $\chi\in \F_q$, $s\in \Z_{\geqslant 1}$ and $B_{q,s,\chi}$ be the set of equivalence classes of $\alpha$'s of degree $s$ over $\F_q$, $(\theta_1$ is related to $\theta_2$ if $\theta_1^{q^i}=\theta_2$ for some $i\in\Z_{\geqslant 0})$ such that ${\rm Tr}_{\F_{q^s}/\F_q}(\alpha)=\chi$ and $F_{ps-1}\equiv 0(\bmod\;T^q-T-\alpha)$. 
\end{defn}
We will often write $\alpha\in B_{q,s,\chi}$ when referring to any representative of the equivalence class of $\alpha$ in $B_{q,s,\chi}$.

\begin{thm}\label{IKL}
For $q\geqslant 3$, $s\geqslant 1$, and $\chi\in \F_q^{\times}$, let $\mathcal{G}=\chi\F_p$. Then the $\mathcal{G}$-fixed c-Wieferich primes of degree $ps$ are precisely the prime factors over $\F_q$ of $R_{q,s,\alpha}=\prod_{i=1}^s(T^q-T-\alpha^{q^i})$ for $\alpha\in B_{q,s,\chi}$. Moreover, these factors are distinct so the number of $\mathcal{G}$-fixed c-Wieferich primes is $p^{-1}q|B_{q,s,\chi}|$.
\end{thm}
\begin{proof}
Let $\mathcal{G}=\chi \F_p$ and $f$ be a $\mathcal{G}$-fixed c-Wieferich prime in $\F_q[T]$ of degree $ps$. By Theorem \ref{LOKI}, we have $f=\prod_{i=1}^s(T^p-\chi^{p-1}T-\gamma^{q^i})$ for some $\gamma\in\F_{q^s}$ of degree $s$ over $\F_q$. Since $f$ is a prime polynomial in $\F_q[T]$ of degree $ps$, it follows that $g:=T^p-\chi^{p-1}T-\gamma^{q^j}$ is irreducible over $\F_{q^s}$. Let $\zeta$ be a root of $g$ (so $\zeta$ will have degree $ps$ over $\F_q$) and $h=\prod_{\varepsilon\in \F_q}(T-(\zeta+\varepsilon))$. Then $h=T^q-T-\alpha$, where $\alpha=\prod_{\varepsilon\in \F_q}(\zeta+\varepsilon)$. If $\zeta^{q^k}=\zeta+\varepsilon$ for some $\varepsilon\in \F_q$, then it follows that $\zeta^{q^{kp}}=\zeta$ which implies that $s$ divides $k$. It follows that $\alpha$ has degree $s$ over $\F_q$, since whenever $\zeta^{q^k}$ is a root of $h$ for some $k>0$, then $s$ divides $k$. Since $\alpha\in \F_{q^s}$, this implies that $\zeta^{q^s}=\zeta+w$, where $w={\rm Tr}_{\F_{q^{s}}/\F_{q}}(\alpha)$. We first show that $w\F_p=\mathcal{G}$. For any $i=0,1,\ldots, p-1$, we have $\zeta^{q^{is}}=\zeta+iw$. Since $g$ is the minimum polynomial of $\zeta$ over $\F_{q^s}$, we have $\gamma^{q^j}=\mathcal{N}_{\F_{q^{ps}}/\F_{q^s}}(\zeta)=\zeta\zeta^{q^s}\cdots \zeta^{q^{s(p-1)}}=\zeta(\zeta+w)\cdots(\zeta+(p-1)w)=\zeta^p-w^{p-1}\zeta$, i.e., $\zeta^p-w^{p-1}\zeta -\gamma^{q^j}=0$. Since $\zeta^p-\chi^{p-1}\zeta=0$, it follows that $w=a\chi$ for some $a\in\F_p^{\times}$ and hence $w\F_p=a\chi \F_p=\chi\F_p=\mathcal{G}$. Now, since $f(\zeta)=0$, we have $F_{ps-1}(\zeta)=0$. This together with the fact $F_{ps-1}$ is $\F_q$-fixed, we obtain $F_{ps-1}(\zeta+\varepsilon)=0$ for any $\varepsilon\in\F_q$ and hence $F_{ps-1}\equiv 0(\bmod\;h)$. With these two results, we can conclude that $\alpha\in B_{q,s,\chi}$ and that $f$ divides $R_{q,s,\alpha}$. Say now that $f$ is a prime factor in $\F_q[T]$ of $R_{q,s,\alpha}$ for some $\alpha\in B_{s,q,\chi}$. Since $\alpha$ has degree $s$ over $\F_q$, it follows that $f$ has degree $ps$ and $f$ will be a $\mathcal{G}$-fixed polynomial by Theorem \ref{BIII?}. By assumption, $F_{ps-1}\equiv 0(\bmod\;R_{q,s,\alpha})$ and so $f$ is a $\mathcal{G}$-fixed c-Wieferich prime of degree $ps$. This finishes the proof of the first statement. If $R_{q,s,\alpha}$ and $R_{q,s,\beta}$ are not coprime (share a root), then $R_{q,s,\alpha}= R_{q,s,\beta}$ because fixing any root $\zeta$ of $R_{q,s,\alpha}$ or $R_{q,s,\beta}$, we can get all the roots by taking $\zeta^{q^i}+\varepsilon$ with $\varepsilon\in \F_q$ and $i=1,\ldots,s$. Since $R_{q,s,\alpha}= R_{q,s,\beta}$ if and only if $\alpha=\beta^{q^i}$ for some $i=1,\ldots,s$, the last statement follows.
\end{proof}

Observe that if $\mathcal{P}$ is a $\mathcal{G}$-fixed prime over $\F_q$ of degree $ps$, then so is $\mathcal{Q}=\mathcal{P}(T+\zeta)$ for any $\zeta\in \F_q-\mathcal{G}$. This is because, for any $\chi\in \mathcal{G}$, we have $\mathcal{Q}(T+\chi)=\mathcal{P}(T+\zeta+\chi)=\mathcal{P}(S+\chi)=\mathcal{P}(S)=\mathcal{P}(T+\zeta)=\mathcal{Q}$, hence $\mathcal{Q}$ is another $\mathcal{G}$-fixed prime. It is clear that $\mathcal{Q}\neq \mathcal{P}$, since $\mathcal{P}$ is a $\mathcal{G}$-fixed prime, hence $\mathcal{P}$ and $\mathcal{Q}$ are in different orbits. Moreover, $\mathcal{Q}$ is unique up to translation by an element of $\mathcal{G}$. So there are $p^{-1}q=\#(\F_q/\mathcal{G})$ possible choices of such $\mathcal{G}$-fixed primes. In other words, if $\mathcal{S}$ is the set of $\mathcal{G}$-fixed c-Wieferich primes of degree $ps$ over $\F_q$, then the other $\frac{q-1}{p-1}-1$ subgroups of $\F_q$ each of order $p$ permute these $\#\mathcal{S}$ distinct c-Wieferich primes. 

We now state the second criterion that characterizes non-fixed c-Wieferich primes in $\F_q[T]$.
\begin{thm}\label{IKL2}
Let $q\geqslant 3$ and $s\geqslant 1$. The non-fixed c-Wieferich primes in $\F_q[T]$ of degree $s$ are precisely the prime factors over $\F_q$ of $R_{q,s,\alpha}=\prod_{i=1}^s(T^q-T-\alpha^{q^i})$ for $\alpha\in B_{q,s,0}$.
\end{thm}
\begin{proof}
Let $f$ be a non-fixed c-Wieferich prime in $\F_q[T]$ of degree $s$. Then $f=\prod_{i=1}^s(T-\gamma^{q^i})$ for some $\gamma\in\F_{q^s}$ of degree $s$ over $\F_q$. Let $\zeta=\gamma^{q^i}$ for one of the $i$'s, $g=T-\zeta$ and define $h=\prod_{\varepsilon\in \F_q}(T-(\zeta+\varepsilon))$. Then $h=T^q-T-\alpha$, where $\alpha=\prod_{\varepsilon\in \F_q}(\zeta+\varepsilon)$. Since $g$ has degree $1$ over $\F_{q^s}$ and divides $h$, it follows by Theorem \ref{BIII?} that ${\rm  Tr}_{\F_{q^s}/\F_q}(\alpha)=0$. It follows that $\alpha$ has degree $s$ over $\F_q$. Since $F_{s-1}(\zeta)=0$ and $F_{s-1}$ is $\F_q$-fixed, it follows that $F_{s-1}(\zeta+\varepsilon)=0$ and hence $F_{s-1}\equiv 0(\bmod\;h)$. We conclude that $\alpha\in B_{q,s,0}$ and that $f$ divides $R_{q,s,\alpha}$. Say now that $f$ is a prime factor in $\F_q[T]$ of $R_{q,s,\alpha}$ for some $\alpha\in B_{s,q,0}$. Since $\alpha$ has degree $s$ over $\F_q$, it follows that $f$ has degree $s$ and $f$ will be a non-fixed polynomial. By hypothesis, $F_{s-1}\equiv 0(\bmod\;T^q-T-\alpha)$ hence $F_{s-1}\equiv 0(\bmod\;R_{q,s,\alpha})$ and so $f$ is a non-fixed c-Wieferich prime. 
\end{proof}

In Table \ref{eval_table2}, we shall see an instance of Theorem \ref{IKL2} in the case where $q=2^4$ and $s=3$.

If there exists some $\mathcal{G}$-fixed c-Wieferich prime $\mathcal{P}\in \F_q[T]$ of degree $pn$, then Theorem \ref{IKL} guarantees existence of an $\alpha\in \F_{q^{n}}$ of degree $n$ over $\F_{q}$ and ${\rm Tr}_{\F_{q^{n}}/\F_q}(\alpha)\neq 0$ such that $F_{pn-1}\equiv 0(\bmod\;T^q-T-\alpha)$. Since $\alpha\in \F_{q^{n}}$, and $[\F_{q^{pn}}:\F_{q^{n}}]=p$, it follows that ${\rm Tr}_{\F_{q^{pn}}/\F_q}(\alpha) = 0$. Therefore, although $\alpha$ may no longer be of degree $pn$ over $\F_{q}$, the congruence relation $F_{pn-1}\equiv 0(\bmod\; T^q-T-\alpha)$ still holds in $\F_{q^{pn}}$. In particular, if the condition in Theorem \ref{IKL2} that $\alpha$ is of degree $s$ over $\F_q$ is relaxed, then Theorem \ref{IKL2} can be used to yield both $\mathcal{G}$-fixed and non-fixed c-Wieferich primes in $\mathcal{A}$. For example, using Theorem \ref{IKL} with $q=3^2$ and $s=1$ yields the same c-Wieferich primes as using Theorem \ref{IKL2} with $q=3^2$ and $s=3$ (with the relaxed condition). 

\section{Computing $\mathcal{G}$-fixed c-Wieferich primes in $\F_{q^{r}}[T]$}\label{D}
We use Theorem \ref{IKL} to develop Algorithm \ref{ALG0-} used to compute $\mathcal{G}$-fixed c-Wieferich primes in $\F_q[T]$, where $\mathcal{G}$ is any subgroup of $\F_q$ of order $p$. For each $i\in \Z_{\geqslant 1}$, we have $[i]=[1]^{q^{i-1}}+\cdots+[1]\equiv \alpha^{q^{i-1}}+\cdots+\alpha(\bmod\;[1]-\alpha)$ over $\F_{q^s}$. In particular, if $q=p$, and $\alpha\in \F_q^{\times}$, then $[1]-\alpha$ is an  Artin-Schreier prime in $\F_q[T]$. Moreover, $[n]\equiv n\alpha(\bmod\; [1]-\alpha)$ for each $n\in \Z_{\geqslant 1}$. This fits well with results described in \cite[Section 5]{ASB2}. 

\begin{algorithm}[ht]
\caption{\; Computing translation invariant c-Wieferich primes in $\F_q[T]$}
\label{ALG0-}
\begin{enumerate}
\item[] \textbf{Input}: $p$ - characteristic of the prime field, $\ell$ - extension degree, $q=p^{\ell}$, $\mathscr{B}$ - a list of elements in $\F_{q^s}$ of degree $s$ over $\F_q$. 
\item[] \textbf{Output}: $\mathscr{F}$- a list of translation invariant c-Wieferich primes in $\F_q[T]$ of degree $ps$.

1. $\mathscr{F}$ and $\mathcal{W}$ are empty lists.

2. for $\alpha$ in $\mathscr{B}$ 

\hspace{1cm} $\beta\longleftarrow \alpha$, $F\longleftarrow -1+\alpha$

\hspace{1cm} for $i=2$ to $ps-2$ 

\hspace{2cm} $\alpha\longleftarrow \alpha^q$, $\beta\longleftarrow \beta + \alpha$, $F\longleftarrow (-1)^i+\beta F$

\hspace{1cm} if $F=0$

\hspace{2cm} $\mathcal{W}\longleftarrow \alpha$ 

3. $w\longleftarrow 1$

4. for $i=1$ to size of $\mathcal{W}$

\hspace{1cm} $w\longleftarrow (T^q-T-\mathcal{W}_i)w$, // $\mathcal{W}_i$ is the $i$th element in $\mathcal{W}$ 

5. $\mathscr{F}\longleftarrow$ prime factors of $w$ as an element in $\F_q[T]$.

\item[] \textbf{Return}: $\mathscr{F}$
\end{enumerate}
\end{algorithm}

A modification of Algorithm \ref{ALG0-} using Theorem \ref{IKL2} yields non-fixed c-Wieferich primes. However, an algorithm based on Theorem \ref{IKL} is much more efficient and faster in computing fixed c-Wieferich primes. In Table \ref{eval_table1a}, we give c-Wieferich primes in $\F_{3^2}[T]$ with corresponding elements (used to compute them) of degree $s$ over 
\begin{table}[ht!]\centering
\small{
\begin{tabular}{ccc}
\hline
$s$ & $\alpha$ in $\F_{9^{s}}$ of degree $s$ over $\F_9$ with ${\rm Tr}_{\F_{9^s}/\F_9}(\alpha)\neq 0$ & c-Wieferich primes \\
\hline
\hline
1  & $\{t, 2t + 1\}$ & $T^3 + (t + 1)T + 1, T^3 + (t + 1)T + t$,\\&& $T^3 + (t + 1)T + 2t + 2, T^3 + (2t + 2)T + 1$,\\ && {\tiny{$T^3 + (2t + 2)T + t + 1,T^3 + (2t + 2)T + 2t + 1$}}.\\
$2,3,4,5,6,7,8$  & -- & -- \\
\hline
\end{tabular}}
\caption{Some elements in subfields of $\overline{\F_{3^2}}$ and the corresponding c-Wieferich primes.}
\label{eval_table1a}
\end{table}
$\F_9$. We do the same for the rings $\F_{5}[T]$ and $\F_{16}[T]$ as indicated in Tables \ref{eval_table4} and \ref{eval_table2} respectively. Observe that, all c-Wieferich primes of degree $2$ and $6$ in $\F_{16}[T]$ are $\F_2$-invariant. The examples highlighted in blue in Table \ref{eval_table3} are $\{0,t\}$-invariant c-Wieferich primes in $\F_{2^{4}}[T]$, while that in red was discovered by V. Mauduit. 

\begin{table}[ht!]\centering
\small{
\begin{tabular}{ccc}
\hline
$s$ & $\alpha$ in $\F_{5^{s}}$ of degree $s$ over $\F_5$ with ${\rm Tr}_{\F_{5^s}/\F_5}(\alpha)\neq 0$ & c-Wieferich primes\\
\hline
\hline
1  & $\{4\}$ & $T^5+4T+1$ \\
2  & $\{2t + 2, 3t + 4\}$ & $T^{10} + 3T^6 + 4T^5 + T^2 + T + 1$ \\
4  & $\{t^3 + t^2 + 4, t + 4, 3t^3 + 3t^2 + t + 4,t^3 + t^2 + 3t + 4,$& {\tiny{$T^{20} + T^{16} + 4T^{15} + T^{12} + 3T^{11} + $}} \\ 
& {$t^3 + t^2 + 3t + 4\}$} & {\tiny{$ + T^8 + 2T^7 + T^5 + T^4 + T^3 + 4T + 1$}}  \\
$3,5,6,7,8,10,11$, & -- & -- \\
\hline
\end{tabular}}
\caption{Some elements in subfields of $\overline{\F_{5}}$ and the corresponding c-Wieferich primes.}
\label{eval_table4}
\end{table}
\begin{table}[ht!]\centering
\small{
\begin{tabular}{ccc}
\hline
$s$ & $\alpha$ in $\F_{16^{s}}$ of degree $s$ over $\F_{16}$ with ${\rm Tr}_{\F_{16^s}/\F_{16}}(\alpha)\neq 0$ & c-Wieferich primes\\
\hline
\hline
1  & {\tiny{$\{1\}$}} & {\tiny{$T^2 + T + t^3, T^2 + T + t^3 + 1, T^2 + T + t^3 + t^2 + t$}},\\&& {\tiny{$T^2 + T + t^3 + t, T^2 + T + t^3 + t^2, T^2 + T + t^3 + t^2 + 1$}},\\ && {\tiny{$T^2 + T + t^3 + t + 1, T^2 + T + t^3 + t^2 + t + 1$}}.\\
2  & -- & -- \\
3  & {\tiny{$\{z^{11} + z^9 + z^6 + z^5 + z^4 + z^2,$}} & {\tiny{$T^6 + T^5 + t^3T^4 + T^3 + (t^2 + 1)T^2 + (t^3 + t^2)T + t^2 + t + 1$,}}\\ & {\tiny{$z^{11} + z^{10} + z^8 + z^7 + z^6 + z^2 + z$}}, & {\tiny{$T^6 + T^5 +
(t^3 + 1)T^4 + T^3 + (t^2 + 1)T^2 + (t^3 + t^2 + 1)T + t + 1$,}}\\& {\tiny{$z^{10} + z^9 +z^8 + z^7 + z^5 + z^4 + z + 1\}$}} & {\tiny{$T^6 + T^5 + (t^3 + t)T^4 + T^3 + (t + 1)T^2 + t^3T + t^2 + t$,}}\\ && {\tiny{$T^6 + T^5 +
(t^3 + t + 1)T^4 + T^3 + (t + 1)T^2 + (t^3 + 1)T + t^2$,}}\\&& {\tiny{$T^6 + T^5 + (t^3 + t^2)T^4 + T^3 + tT^2 + (t^3 + t^2 + t + 1)T + t^2 + t$.}}\\& & {\tiny{$T^6 +T^5 + (t^3 + t^2 + 1)T^4 + T^3 + tT^2 + (t^3 + t^2 + t)T + t^2 + 1$,}}\\&& {\tiny{$T^6 + T^5 + (t^3 + t^2 + t)T^4 + T^3 + t^2T^2 + (t^3 + t + 1)T + t$.}}\\&&{\tiny{$T^6 + T^5 + (t^3 + t^2 + t + 1)T^4 + T^3 + t^2T^2 + (t^3 + t)T + t^2 + t + 1$.}}\\
\hline
\end{tabular}}
\caption{Some elements in subfields of $\overline{\F_{16}}$ and the corresponding c-Wieferich primes.}
\label{eval_table2}
\end{table}

We demonstrate how to compute the c-Wieferich primes of degree $3$ over $\F_{2^4}$. To determine these non-fixed c-Wieferich primes of degree $3$ in $\F_{2^{4}}[T]$ according to Theorem \ref{IKL2}, we set the degree of the extension to be $s=3$ and search all the elements $\alpha\in \F_{2^{12}}$ with ${\rm Tr}_{\F_{2^{12}}/\F_{2^4}}(\alpha)=0$ and $F_{2}\equiv 0(\bmod\;T^{2^4}-T-\alpha)$. For example, we used the cubic field extension $\F_{2^4}(z)$ with $f_{\rm min}^z(X)=X^3 + X + 1$. Using SAGE, we found $B_{2^4,3,0}=\{\alpha\}=\{\alpha_1=z^{11} + z^9 + z^6 + z^5 + z^4 + z^2 + 1, \alpha_2=z^{10} + z^9 + z^8 + z^7 + z^5 + z^4 + z, \alpha_3=z^{11} + z^{10} + z^8 + z^7 + z^6 + z^2 + z + 1\}$. So $R_{2^4,3,0}=(T^{2^4}-T-\alpha_1)(T^{2^4}-T-\alpha_2)(T^{2^4}-T-\alpha_3)=T^{48} + T^{33} + T^{18} + T^{16} + T^3 + T + 1\in \F_{2^4}[T]$. Upon factorisation of $R_{2^4,3,0}\in \F_{2^4}[T]$, we obtain the list of c-Wieferich primes of degree $3$ indicated in the third column of Table \ref{eval_table3}, (with $f_{\rm min}^t(X)=X^4+X+1$). Furthermore, since $T^3+T+1$ is a non-fixed c-Weiferich prime, there are $15$ more non-fixed c-Weiferich primes obtained by applying a translation $\F_{2^4}$-automorphism to $T^3+T+1$, for example, $\sigma_t(T^3+T+1)=(T+t)^3+(T+t)+1=T^3 + tT^2 + (t^2 + 1)T + t^3 + t + 1$. We summarise these results in Table \ref{eval_table3}. Observe the similarity in the primes of Tables \ref{eval_table2} and \ref{eval_table3}. We also observe that all the c-Wieferich primes of degree $2$ and $6$ in $\F_{2^4}[T]$ are $\F_2$-invariant. The examples highlighted in blue are $\{0,t\}$-invariant c-Wieferich primes in $\F_{2^{4}}[T]$, while that in red was discovered by V. Mauduit, \cite{veronique}.

\begin{table}[ht!]\centering
\small{
\begin{tabular}{ccc}
\hline
$s$ & $\alpha$ in $\F_{16^{s}}$ of degree $s$ over $\F_{16}$ with ${\rm Tr}_{\F_{16^s}/\F_{16}}(\alpha)=0$ & c-Wieferich primes of degree $s$ in $\mathcal{A}$\\
\hline
\hline
1  & -- & -- \\
2  & $\{1\}$ & $T^2 + T + t^3 + t^2, T^2 + T + t^3 + 1$,\\ && $ T^2 + T + t^3 + t, T^2 + T + t^3 + t^2 + 1,$\\
  &  & $T^2 + T + t^3,  T^2 + T + t^3 + t^2 + t$, \\ & & $T^2 + T + t^3 + t^2 + t + 1, T^2 + T + t^3 + t + 1$ \\ 
3  & $\{z^{11} + z^9 + z^6 + z^5 + z^4 + z^2 + 1$,  & ${\color{blue}{T^3 + T + 1}}, {\color{blue}{T^3 + tT^2 + (t^2 + 1)T + t^3 + t + 1}}$, \\ & $z^{11} + z^{10} + z^8 + z^7 + z^6 + z^2 + z + 1$,  & ${\color{red}{T^3 + T^2 + 1}},T^3 + t^2T^2 + tT + t^3 + 1$, \\ & $z^{10} + z^9 + z^8 + z^7 + z^5 + z^4 + z\}$ & $ T^3 + (t^2 + 1)T^2 + (t + 1)T + t^3 + t^2 + t$, \\& & $T^3 + (t^2 + t + 1)T^2 + (t^2 + t + 1)T + t^2 + t + 1$, \\ & & $T^3 + (t^3 + 1)T^2 + (t^3 + t^2)T + t^2 + t + 1$, \\&& $T^3 + (t^3 + t + 1)T^2 + t^3T + t^2 + t$, \\&& $T^3 + (t^3 + t)T^2 + (t^3 + 1)T + t^2$, \\&& $T^3 + (t^3 + t^2)T^2 + (t^3 + t^2 + t)T + t^2 + 1$, \\&& $T^3 + (t^3 + t^2 + 1)T^2 + (t^3 + t^2 + t + 1)T + t^2 + t$, \\ & & $ T^3 + (t^3 + t^2 + t)T^2 + (t^3 + t)T + t^2 + t + 1$, \\&& $T^3 + (t^2 + t)T^2 + (t^2 + t)T + t^2 + t$,\\&& $T^3 + t^3T^2 + (t^3 + t^2 + 1)T + t + 1$,\\&& $T^3 + (t + 1)T^2 + t^2T + t^3 + t^2 + 1$,\\ & & $T^3 + (t^3 + t^2 + t + 1)T^2 + (t^3 + t + 1)T + t$\\
5&--&--\\
6& {\tiny{$\{z^{20} + z^{19} + z^{18} + z^{16} + z^{15} + z^{13} + z^{11}
+ z^{10} + z^6 + z^2 + 1$,}} & product of c-Wieferich primes of degree $6$ is \\& {\tiny{$z^{23} + z^{21} + z^{19} + z^{15} + z^{14} + z^{13} + z^{11} + z^{10} + z^9 + z^7 + z^4 + z^3 + 1$,}} &$T^{48} + T^{33} + T^{32} + T^{18} + T^3 + T^2 + 1$\\& {\tiny{$z^{23} + z^{21} + z^{20} + z^{18} + z^{16} + z^{14} + z^9 + z^7 + z^6 + z^4 + z^3 + z^2 + 1\}$}}&see Table \ref{eval_table2}\\
\hline
\end{tabular}}
\caption{Some elements in subfields of $\overline{\F_{2^4}}$ and the corresponding c-Wieferich primes.}
\label{eval_table3}
\end{table}

From literature, there are three c-Wieferich primes of degrees $5$, $10$ and $20$ defined over $\F_5[T]$ obtained by brute force / other methods. Using Theorem \ref{IKL}, we searched for c-Wieferich primes of degree up to 55 and obtained a new c-Wieferich prime of degree 45. In particular, $T^{45} + T^{41} + 3T^{40} + T^{37} + T^{36} + 2T^{35} + T^{33} + 4T^{32} + T^{31} + T^{30} + T^{29} + 2T^{28} + 2T^{27} + 4T^{26} + 4T^{21} + 3T^{20} + 4T^{17} + T^{15} + 4T^{13} + 2T^{12} + 4T^9 + 4T^8 + 2T^7 + 4T^6 + T^5 + T^4 + 2T^3 + T^2 + 3T + 2\in \F_5[T]$ represented by the triple $(5,1,9)$ highlighted in Table \ref{datatable}. 

We now give some of the ``highest" degree examples of c-Wieferich primes (summarised in Table \ref{datatable}) we obtained in different rings for $p<20$. We found all of the possible c-Wieferich primes in the listed cases, but we only present a few of them. In $\F_{3^4}[T]$, we found $T^9 + T^6 + T^4 + T^3 + T^2 + T + 1,\ldots, T^9 + T^6 + T^4 + (2t^3 + 2t^2 + 2t + 1)T^3 + T^2 + (t^3 + t^2 + t + 1)T + t^3 + 2t + 2$ where $t^4+2t^3+2=0$. In the ring $\F_{7^3}[T]$, we found $T^{14} + 5T^8 + 4T^7 + T^2 + 3T + 6,\ldots, T^{14} + 5T^8 + (6t^2 + 6t)T^7 + T^2 + (t^2 + t)T + t^2 + 6t + 6$ where $t^3+6t^2+4=0$. In $\F_{11^2}[T]$, we obtained $T^{33} + (5t + 3)T^{23} + 10T^{22} + (3t + 1)T^{13} + (4t + 9)T^{12} + (8t + 4)T^{11} + (7t + 8)T^3 + (10t + 7)T^2 + (2t + 3)T + 3t + 9,\ldots, T^{33} + (6t + 1)T^{23} + (10t + 2)T^{22} + (8t + 2)T^{13} + (6t + 2)T^{12} + (3t + 7)T^{11} + (4t + 3)T^3 + (5t + 3)T^2 + (6t + 5)T + 4t + 1$ where $t^2+7t+2=0$. Lastly, in $\F_{19^2}[T]$, we obtained $T^{38} + (8t + 10)T^{20} + 18T^{19} + (18t + 12)T^2 + (15t + 14)T + 16t + 1,\ldots, T^{38} + (11t + 18)T^{20} + (18t + 11)T^{19} + (t + 11)T^2 + (8t + 15)T + 12t + 18\in \F_{19^2}[T]$ where $t^2+18t+2=0$. 

There are some more examples not summarised in the Table \ref{datatable}, these include $T^{129} + 40T^{87} + 7T^{86} + 3T^{45} + 29T^{44} + 17T^{43} + 42T^3 + 7T^2 + 26T + 10$, $T^{129} + 40T^{87} + 22T^{86} + 3T^{45} + 42T^{44} + 28T^{43} + 42T^3 + 22T^2 + 15T + 23$ in $\F_{43}[T]$. In $\F_{61}[T]$, we found $T^{183}+58T^{123}+52T^{122}+3T^{63}+18T^{62}+35 T^{61}+60T^3+52T^2+26T+35$. In $\F_{131}[T]$, we found $T^{393}+128T^{263}+129T^{262}+3T^{133}+4T^{132}+50 T^{131}+130T^3+129T^2+81T+33$. In $\F_{463}[T]$, we found $T^{926}+461T^{464}+193T^{463}+T^2+270T+277$ and $T^{926}+461T^{464}+188T^{463}+T^2+275T+182$.

\section{c-Wieferich primes under constant field extensions}\label{E}
In characteristic $2$, the polynomial $f=T^3+T^2+1$ is always irreducible over $\F_{2^{3m\pm 1}}$ for any $m\in\Z_{\geqslant 1}$. This is because $f$ is a prime polynomial in $\F_2[T]$ and its degree over $\F_2$ is always coprime to $3m\pm 1$ for any $m\in \Z_{\geqslant 1}$. In \cite{veronique}, V. Mauduit showed that $T^3+T^2+1$ is a non-fixed c-Wieferich prime in $\F_{2^4}[T]$. In addition, Mauduit proved that the polynomial $T^3+T^2+1\in \F_{2}[T]$ is a non-fixed c-Wieferich prime for all rings of the form $\F_{2^{3m+1}}[T]$, $m\in\Z_{\geqslant 0}$. However, in her list, \textit{Mauduit left out a number of possible candidates}, see Table \ref{eval_table3} for more examples.  We now generalise Mauduit's horizontal existence result as follows.

\begin{thm}\label{diagonalcw}
Let $\mathcal{P}$ be a prime polynomial in $\F_{q}[T]$ of degree $m$. For any $n\in \Z_{\geqslant 1}$, we have that $\mathcal{P}$ is a c-Wieferich prime in $\F_{q}[T]$, if and only if $\mathcal{P}$ is a c-Wieferich prime in $\F_{q^{mn+1}}[T]$.
\end{thm}
Before we state the proof Theorem \ref{diagonalcw}, we shall need the following extra notation. For each $i,j\in \Z_{\geqslant 1}$, we define the symbols $[i]_{q^j}$ and $F_{q^j,i}$ as follows: $[i]_{q^j}=T^{q^{ji}}-T$, and $F_{q^j,i}=(-1)^i+[i]_{q^j}F_{q^j,i-1}$ with $F_{q^j,0}=1$.
\begin{proof}
Let $r=mn+1$ where $n\in\Z_{\geqslant 0}$ and $\mathcal{Q}_m$ be any prime polynomial in $\F_q[T]$ of degree $m$. Since $m$ is coprime to $r$, we have $\mathcal{Q}_m$ is a prime polynomial in $\F_{q^r}[T]$, hence $[m]_q=T^{q^m}-T\equiv 0(\bmod\;\mathcal{Q}_m)$. It follows that, for any $i\in \Z_{\geqslant 1}$, we have $[i]_{q^{r}}=T^{q^{r i}}-T=T^{q^{(mn+1)i}}-T=T^{q^{mni}\cdot q^i}-T\equiv T^{q^{m}\cdots q^{m}\cdot q^i}-T\equiv T^{q^i}-T=[i]_q(\bmod\;\mathcal{Q}_m)$. Therefore, $F_{q^{r},m-1}\equiv F_{q,m-1}(\bmod\;\mathcal{Q}_m)$ and the claim follows from this result.
\end{proof}

We know that $f=T^3+T^2+1$ and $f(T+1)=T^3+T+1$ are non-fixed c-Wieferich primes in $\F_2[T]$. By Theorem \ref{diagonalcw}, it follows that they are c-Wieferich primes in $\F_{2^{3+1}}[T]$, (see Table \ref{eval_table3}). Another example is the polynomial $\mathcal{P}=T^6+T^4+T^3+T^2+2T+2$ which is a fixed c-Wieferich prime for $\F_3[T]$. By Theorem \ref{diagonalcw}, $\mathcal{P}$ is a c-Wieferich prime for $\F_q[T]$ where $q=3^{6m+1}$ and $m\in\Z_{\geqslant 1}$. On the other hand, the fixed polynomial $\mathcal{P}_1={\color{red}{T^6 + T^4 + 2T^3 + T^2 + T + 2}}=2^{-\deg(\mathcal{P})}\mathcal{P}(2T)$ is not a c-Wieferich prime polynomial in $\F_{3}[T]$ but in $\F_{3^5}[T]$. By Theorem \ref{diagonalcw}, the polynomial $\mathcal{P}_1$ is also a c-Weferich prime in the rings of the form $\F_{3^{5(6n+1)}}[T]$ for any $n\in\Z_{\geqslant 0}$. Furthermore, since the Diophantine equation $6n+1=30k+5$ has no integer solutions, it follows that there is no ring extension $\F_{3^{r}}[T]$ in which both $\mathcal{P}$ and $\mathcal{P}_1$ occur as c-Wieferich primes. 

To understand the behaviour above, consider the triple $(p,m,s)$, where $p$ is a prime number and $s,m\in \Z_{\geqslant 1}$. We say $\mathcal{P}$ is a c-Wieferich prime for the triple $(p,m,s)$ if there exists a subgroup $\mathcal{G}$ of $(\F_{p^{m}},+)$ such that $\mathcal{P}$ is a $\mathcal{G}$-fixed c-Wieferich prime in $\F_{p^m}[T]$ of degree $ps$. For example, $\mathcal{P}=T^6+T^4+T^3+T^2+2T+2$ is a c-Wieferich prime for the triples $(3,1,2)$, $(3,7,2)$, $(3,13,2),\ldots$ since $\mathcal{P}$ is an $\F_3$-fixed c-Wieferich prime of degree $6$ in $\F_{3}[T]$, $\F_{3^7}[T]$ and $\F_{3^{13}}[T]$. For each $q$ and $s\in\Z_{\geqslant 1}$, we let $B_{q,s}=\cup_{\chi\in \F_q^{\times}}B_{q,s,\chi}$. Each element of $B_{q,s}$ has $p^{-1}q$ distinct c-Wieferich primes associated to it. We now summarise some of the data we have computed regarding counts of c-Wieferich primes associated to the triples $(p,m,s)$ in Table \ref{datatable}. The values in the table are $(p,m,s)$ and $|B_{p^m,s}|$, and they give at least some indication about how these numbers vary.

\begin{table}[ht!] \centering \caption{Counts of fixed c-Wieferich primes}\label{datatable}
$
\begin{array}{|c|c||c|c||c|c||c|c||c|c||c|c|}
\hline
(p,m,s) & \# & (p,m,s) &  \# & (p,m,s) & \# & (p,m,s) &\# & (p,m,s) & \# & (p,m,s) &\# \\
\hline\hline 
(3,1,1) & 0 & (3,1,2) & 1 & (3,1,3) & 1 & (3,1,4) & 1 & (3,1,5) & 1 & (3,1,6) & 0 \\ \hline
(3,1,7) & 0 & (3,1,8) & 0 & (3,1,9) & 0 & (3,1,10) & 0 & (3,1,11) & 0  & (3,1,12) & 0 \\ \hline
(3,1,13) & 0 & (3,1,14) & 0 & (3,1,15) & 0 & (3,1,16) & 0 & (3,1,17) & 0 & (3,2,1) & 2 \\ \hline
(3,2,2) & 0 & (3,2,3) & 0 & (3,2,4) & 0  & (3,2,5) & 0  & (3,2,6) & 0 & (3,2,7) & 0 \\ \hline
(3,2,8) & 0 & (3,3,1) & 0 & (3,3,2) & 1 & ( (3,3,3) & 0 & (3,3,4) & 0 &(3,3,5) & 0 \\ \hline
(3,4,1) & 2 & (3,4,2) & 0 & (3,4,3) & 1 & (3,4,4) & 0 & (3,5,1) & 0 & (3,5,2) & 1 \\ \hline
(3,5,3) & 0 & (3,6,1) & 2 & (3,6,2) & 0  & (3,7,1) & 0 & (3,7,2) & 1 & (3,8,1) & 2 \\ \hline
(3,8,2) & 0 & (3,9,1) & 0 & (3,10,1) & 2 & (3,11,1) & 0 & (3,12,1) & 2 & (3,13,1) & 0 \\ \hline
(3,14,1) & 2 & (3,15,1) & 0 & (3,16,1) & 2 & (3,17,1) & 0 & (5,1,1) & 1 & (5,1,2) & 1 \\ \hline
(5,1,3) & 0 & (5,1,4) & 1 & (5,1,5) & 0 & (5,1,6) & 0 & (5,1,7) & 0 & (5,1,8) & 0 \\ \hline
\color{blue}{{(5,1,9)}} & 1 & (5,1,10) & 0 & (5,1,11) & 0 & (5,2,1) & 1 & (5,2,2) & 0 & (5,2,3) & 0 \\ \hline
(5,2,4) & 0 & (5,2,5) & 0 & (5,3,1) & 4 & (5,3,2) & 1 & (5,3,3) & 0 & (5,4,1) & 1 \\ \hline
(5,4,2) & 0 & (5,5,1) & 1 & (5,5,2) & 1 & (5,6,1) & 4 & (5,7,1) & 1 & (5,8,1) & 1 \\ \hline
(5,9,1) & 4 & (5,10,1) & 1 & (5,11,1) & 1 & (7,1,1) & 1 & (7,1,2) & 1 & (7,1,3) & 0 \\ \hline
(7,1,4) & 0 & (7,1,5) & 0 & (7,1,6) & 0 & (7,1,7) & 0 & (7,1,8) & 0 & (7,1,9) & 0 \\ \hline
(7,2,1) & 3 & (7,2,2) & 0 & (7,2,3) & 1 & (7,2,4) & 0 & (7,3,1) & 4 & (7,3,2) & 1 \\ \hline
(7,3,3) & 3 & (7,4,1) & 3 & (7,4,2) & 0 & (7,5,1) & 1 & (7,6,1) & 6 & (7,7,1) & 1 \\ \hline
(7,8,1) & 3 & (7,9,1) & 4 & (11,1,1) & 1 & (11,1,2) & 0 & (11,1,3) & 1 & (11,1,4) & 0 \\ \hline
(11,1,5) & 0 & (11,1,6) & 0 & (11,1,7) & 0 & (11,2,1) & 1 & (11,2,2) & 2 & (11,2,3) & 2 \\ \hline
(11,3,1) & 1 & (11,3,2) & 0 & (11,4,1) & 5 & (11,5,1) & 6 & (11,6,1) & 1 & (11,7,1) & 1 \\ \hline
\end{array}
$
\end{table}

For example, for the triple $(3,2,1)$, we have $B_{3^2,1}=\{t,2t+1:t^2+2t+2=0\}$ and hence $3\cdot 2=6$ different fixed c-Wieferich primes of degree $3$, namely $\mathcal{L}=\{T^3 + (t + 1)T + 1, T^3 + (t + 1)T + t, T^3 + (t + 1)T + 2t + 2, T^3 + (2t + 2)T + 1, T^3 + (2t + 2)T + t + 1, T^3 + (2t + 2)T + 2t + 1\}$, (also see Table \ref{eval_table1a}). By Theorem \ref{diagonalcw}, these same primes correspond to the triples $(3,2(3n+1),1)$ where $n\in \Z_{\geqslant 0}$. We demonstrate this below: 
\begin{enumerate}
\item in $\F_{3^4}[T]$, we have $B_{3^4,1}=\{2w^3 + 2w^2 + 1, w^3 + w^2:w^4 + 2w^3 + 2=0\}$. This is the image of $B_{3^2,1}$ under the embedding $\iota:\F_{3^2}\to \F_{3^4}$ defined by $t\mapsto w^3 + w^2$. So the primes in $\mathcal{L}$ again appear among c-Wieferich primes in $\F_{3^4}[T]$, (in this case there are $3^{-1}\cdot 3^4\cdot 2$ distinct c-Wieferich primes in $\F_{3^4}[T]$ of degree $3$). For example, $T^3 + (t + 1)T + 1$ is identified with $T^3 + (w^3 +w^2 + 1)T + 1$ in $\F_{3^4}[T]$. 
\item in $\F_{3^8}[T]$, we have $B_{3^8,1}=\{2w^7 + 2w^6 + 2w^5 + w + 2, w^7 + w^6 + w^5 + 2w + 2:w^8 + 2w^5 + w^4 + 2w^2 + 2w+ 2=0\}$. This is the image of $B_{3^2,1}$ under $\iota:\F_{3^2}\to \F_{3^8}$ defined by $t\mapsto w^7 + w^6 + w^5 + 2w + 2$. 
\end{enumerate}

Attached to the triple $(3,5,2)$ is a c-Wieferich prime $\mathcal{Q}=T^6 + T^4 + 2T^3 + T^2 + T + 2$. One can easily show that $\mathcal{Q}$ is one of the two fixed primes in $\F_3[T]$ of degree $6$, (the other being $\mathcal{P}=T^6 + T^4 + T^3 + T^2 + 2T + 2$). Furthermore, we have already seen that $\mathcal{P}$ is a c-Wieferich prime in $\F_3[T]$ while $\mathcal{Q}$ is a c-Wieferich prime in $\F_{3^5}[T]$. With this limited data, a natural question arises: given a $q$, do all fixed primes in $\F_{q}[T]$ become c-Wieferich primes in some $\F_{q^{r}}[T]$? \textit{We provide a partial answer in Lemma \ref{impo} and Theorem \ref{impo3}.}

\begin{lem}\label{impo}
Fix $s\in \Z_{\geqslant 1}$ and an element $\alpha \in \F_{q^s}$ of degree $s$ over $\F_q$. Fix $k\in \Z_{\geqslant 1}$ and put $r=1+sk$. Then for any $n\in \Z_{\geqslant 1}$, we have that $F_{q,n}\equiv 0(\bmod\;T^q-T-\alpha)$ if and only if  $F_{q^r,n}\equiv 0(\bmod\;T^{q^r}-T-\alpha)$. In particular, for a fixed $\chi\in \F_q^{\times}$, we have $\alpha\in B_{q^r,s,\chi}$ if and only if $\alpha\in B_{q,s,\chi}$.
\end{lem}
\begin{proof}
Put $f_{q,0}=1$ and define $f_{q,i}\in \F_q[x]$ recursively as follows $f_{q,i}(x)=(-1)^i+(\sum_{j=0}^{i}x^{q^j})f_{q,i-1}(x)$. Since $\alpha\in \F_{q^s}$, we have $\alpha^{q^r}=(\alpha^{q^{sk}})^q=\alpha^q$ and so $f_{q,n}(\alpha)=f_{q^r,n}(\alpha)$. Since $F_{q,n}\equiv f_{q,n}(\alpha)(\bmod\;T^q-T-\alpha)$ and $F_{q^r,n}\equiv f_{q^r,n}(\alpha)(\bmod\;T^{q^r}-T-\alpha)$, the first equivalence then follows. For the second equivalence, it remains only to note that $\alpha$ also has degree $s$ over $\F_{q^r}$ and ${\rm Tr}_{\F_{q^{rs}}/\F_{q^r}}(\alpha)={\rm Tr}_{\F_{q^{s}}/\F_{q}}(\alpha)$.
\end{proof}

\begin{thm}\label{impo3}
Let $s \in \Z_{\geqslant 1}$ be fixed, $\mathcal{G}=\chi\F_p$ for some $\chi\in\F_q^{\times}$ and $E_{q,s}=\{k\in \N:p\nmid 1+sk\}$. If there is
\begin{enumerate}
\item a $\mathcal{G}$-fixed c-Wieferich prime $\mathcal{P}$ over $\F_q$ of degree $ps$, then this prime polynomial is a factor of $R_{q,s,\alpha}$ for some $\alpha\in \F_{q^s}$ of degree $s$ over $\F_q$ and ${\rm Tr}_{\F_{q^s}/\F_q}(\alpha)\neq 0$. For each $k\in E_{q,s}$, there exists a $\mathcal{G}$-fixed prime over $\F_q$ that is a prime factor of $R_{q,s,\beta}$, where $\beta=\alpha-\frac{k}{1+sk}{\rm Tr}_{\F_{q^s}/\F_q}(\alpha)$. Moreover, this prime factor is a c-Wieferich prime polynomial in $\F_{q^{1+sk}}[T]$. 
\item no $\mathcal{G}$-fixed c-Wieferich prime over $\F_q$ of degree $ps$, then none of the $\mathcal{G}$-fixed primes over $\F_q$ is a c-Wieferich prime in $\F_{q^r}[T]$ for any $r\in\Z_{\geqslant 1}$. 
\end{enumerate}
\end{thm}
\begin{proof}
Let $\mathcal{P}$ be a $\mathcal{G}=\chi\F_p$-fixed c-Wieferich prime in $\F_q[T]$ of degree $ps$. By Theorem \ref{IKL}, $\mathcal{P}$ is a factor of $R_{q,s,\alpha}$ for some $\alpha$ of degree $s$ over $\F_q$ and $\chi={\rm Tr}_{\F_{q^s}/\F_q}(\alpha)\neq 0$. In particular, $F_{q,ps-1}\equiv 0(\bmod\;T^q-T-\alpha)$. Take any $k\in E_{q,s}$ and put $r=1+sk$. By Lemma \ref{impo}, $F_{q^r,ps-1}\equiv 0(\bmod\;T^{q^r}-T-\alpha)$. Since $r$ is coprime to $s$, we have $\alpha$, (as an element of $\F_{q^{rs}}$) has degree $s$ over $\F_{q^r}$. So ${\rm Tr}_{\F_{q^{rs}}/\F_{q^{r}}}(\alpha)={\rm Tr}_{\F_{q^s}/\F_q}(\alpha)$. 

Let $\beta=\alpha-\frac{k}{(1+sk)}{\rm Tr}_{\F_{q^s}/\F_q}(\alpha)$. Clearly $\beta$ belongs to $\F_{q^s}$. We now show that $\beta$ gives rise to c-Wieferich primes over $\F_{q^r}$ of degree $ps$. First, we show that $\beta$ has degree $s$ over $\F_q$ and ${\rm Tr}_{\F_{q^s}/\F_q}(\beta)\neq 0$. Assume that $\beta$ has degree less than $s$, i.e., $\beta^{q^{ri}}=\beta^{q^{rj}}$ for some $1\leqslant i<j< s$. Then $\alpha^{q^{ri}}-\frac{k}{(1+sk)}{\rm Tr}_{\F_{q^s}/\F_q}(\alpha)=\alpha^{q^{rj}}-\frac{k}{(1+sk)}{\rm Tr}_{\F_{q^s}/\F_q}(\alpha)$ which implies that $\alpha^{q^{ri}}=\alpha^{q^{rj}}$, a contradiction to the fact $\alpha$ has degree $s$ over $\F_{q^r}$. Furthermore, we have ${\rm Tr}_{\F_{q^s}/\F_q}(\beta)={\rm Tr}_{\F_{q^s}/\F_q}(\alpha)-\frac{ks}{1+sk}{\rm Tr}_{\F_{q^s}/\F_q}(\alpha)={\rm Tr}_{\F_{q^s}/\F_q}(\alpha)\neq 0$.

If $w\in \F_{q^{s}}$, then the trinomial $T^q-T-w$ divides $T^{q^r}-T-\alpha$ if and only if $\alpha=w+k{\rm Tr}_{\F_{q^s}/\F_q}(w)$. This implies that ${\rm Tr}_{\F_{q^s}/\F_q}(w)=\frac{1}{1+sk}{\rm Tr}_{\F_{q^s}/\F_q}(\alpha)$, hence $w=\alpha-\frac{k}{(1+sk)}{\rm Tr}_{\F_{q^s}/\F_q}(\alpha)=\beta$. As such, it follows that the polynomial $R_{q,s,\beta}$ divides $R_{q^r,s,\alpha}$. By Theorem \ref{IUO}, the trinomial $T^q-T-\beta$ decomposes into $\mathcal{G}=\chi\F_p$-fixed primes over $\F_{q^s}$. As such, the prime factors of $R_{q,s,\beta}$ over $\F_{q}$ are $\mathcal{G}$-fixed primes of degree $ps$. In addition, these are also c-Wieferich primes in $\F_{q^r}[T]$, since they divide $F_{q^r,ps-1}$. Note that if $p$ divides $1+sk$, then $T^q-T-w$ does not divide $T^{q^r}-T-\alpha$ for any $w\in \F_{q^s}$ since ${\rm Tr}_{\F_{q^s}/\F_q}(\alpha)\neq 0$. 

For the second part, let $\mathcal{P}\in \F_q[T]$ be a $\mathcal{G}$-fixed prime of degree $ps$. Assume that $\mathcal{P}$ is a c-Wieferich prime in $\F_{q^r}[T]$ for some $r=1+sk$ with $k\in E_{q,s}-\{0\}$ and that there are no fixed c-Wieferich primes in $\F_q[T]$ of degree $ps$. By Theorem \ref{IKL}, there exists an $\alpha\in \F_{q^{rs}}^{\times}$ of degree $s$ over $\F_{q^r}$ with ${\rm Tr}_{\F_{q^{rs}}/\F_{q^r}}(\alpha)\neq 0$ such that $F_{q^r,ps-1}\equiv 0(\bmod\;T^{q^r}-T-\alpha)$, i.e., $\mathcal{P}$ divides $\mathcal{R}_{q^r,s,\alpha}$. Since $\mathcal{P}$ is a $\mathcal{G}$-fixed of degree $ps$, it follows from Theorem \ref{LOKI} that it factors into almost Artin-Schreier primes in $\F_{q^s}[T]$. Say that one is $\mathcal{Q}=T^p-\beta T-\lambda$, with $\beta=\chi^{p-1}$ for some $\chi\in \F_q$. We have $\mathcal{Q}$ divides $T^{q^r}-T-\alpha$ and then we see that $\prod_{a\in \F_p}\mathcal{Q}(T+a\chi)=T^q-T-w$ divides $T^{q^r}-T-\alpha$ where $w\in \F_{q^s}$. The polynomial $T^q-T-w$ divides $T^{q^r}-T-\alpha$ if and only if $w=\alpha-\frac{k}{1+sk}{\rm Tr}_{\F_{q^s}/\F_q}(\alpha)$, hence $\alpha\in\F_{q^s}$. Lemma \ref{impo} then tells us that there exists a $\mathcal{G}$-fixed c-Wieferich prime of degree $ps$ in $\F_q[T]$ which is a contradiction.
\end{proof}

Given a $\mathcal{G}$-fixed c-Wieferich prime $\mathcal{P}$ in $\F_q[T]$ of degree $ps$ and $k\in E_{q,s}$, how do we find the $\mathcal{G}$-fixed primes $\mathcal{Q}$ in $\F_q[T]$ that are c-Wieferich primes in $\F_{q^{1+sk}}[T]$? Suppose that the prime $\mathcal{P}$ comes from $\alpha\in B_{q,s,\chi}$ and that $\mathcal{P}$ divides $R_{q,s,\alpha}=\prod_{i=1}^s(T^{q}-T-\alpha^{q^i})$. From the proof of Theorem \ref{impo3}, $\mathcal{Q}$ is a prime factor of $R_{q,s,\beta}$. Now, 
\begin{align*}
R_{q,s,\beta}=\prod_{i=1}^s(T^{q}-T-\beta^{q^i})=\prod_{i=1}^s\left(T^{q}-T-\alpha^{q^i}+\frac{k}{(1+sk)}{\rm Tr}_{\F_{q^s}/\F_q}(\alpha)\right)=\prod_{i=1}^s((T+x)^{q}-(T+x)-\alpha^{q^i}),
\end{align*} 
where $x$ is any root of $X^q-X-\frac{k}{(1+sk)}{\rm Tr}_{\F_{q^s}/\F_q}(\alpha)=0$. By Theorem \ref{BIII?}, we have that $x\in \F_{q^{p}}$, (and the set of solutions of $X^q-X-\frac{k}{(1+sk)}{\rm Tr}_{\F_{q^s}/\F_q}(\alpha)=0$ is $x+\F_q$). Comparing the factorisation of $R_{q,s,\beta}$ to that of $R_{q,s,\alpha}$, we see that $R_{q,s,\beta}$ has $p^{-1}q$ distinct prime factors of the form $\mathcal{P}(T+x)$, and $\mathcal{Q}$ is one of them. Observe that, if $p$ is coprime to $s$, then the set map $f_s:E_{q,s}\to\F_p-\{s^{-1}\}$ defined by $k\mapsto \frac{k}{1+sk}$ is surjective, and if $p$ divides $s$, then the set map $f_s:E_{q,s}\to\F_p-\{0\}$ defined by $k\mapsto \frac{k}{1+sk}$ is surjective. So we only need $p-1$ distinct values of $k$ in $E_{q,s}$ to determine all the possible roots to $X^q-X-\frac{k}{(1+sk)}{\rm Tr}_{\F_{q^s}/\F_q}(\alpha)=0$, for a fixed $\alpha$. For this matter, we need only $D_{q,s}=\{k\in\Z_{\geqslant 0}:p\nmid (1+sk), k\leqslant p-1\}$ instead of the full set $E_{q,s}$.

The polynomials $\mathcal{P}=T^6+T^4+T^3+T^2+2T+2$ and $\mathcal{Q}=T^6 + T^4 + 2T^3 + T^2 + T + 2$ are both fixed polynomials in $\F_{3}[T]$, in fact the only fixed prime polynomials of degree $6$. We have already seen that $\mathcal{P}$ is a c-Wieferich prime in $\F_{3}[T]$ while $\mathcal{Q}$ is not. Furthermore, $\mathcal{Q}$ is a c-Wieferich prime in $\F_{3^5}[T]$ while $\mathcal{P}$ is not. In fact, there exists no ring $\F_{3^m}[T]$ in which both $\mathcal{P}$ and $\mathcal{Q}$ are c-Wieferich primes simultaneously. In any case, here we have $E_{3,2}=\{0,2\}$. For $r=5$, solving $X^3-X-2\cdot 5^{-1}{\rm Tr}_{\F_{3^2}/\F_{3}}(\alpha)=0$ yields $X=z$, where $z^3+2z+1=0$ as a solution. So $\mathcal{P}(T+z)=T^6+2z^3T^3+z^6+T^4+T^3z+Tz^3+z^4+T^3+z^3+T^2+2zT+z^2+2T+2z+2=T^6+T^4+(2z^3+z+1)T^3+T^2+(z^3+2z+2)T+z^6+z^4+z^3+z^2+2z+2=T^6 + T^4 + 2T^3 + T^2 + T + 2$. 

Also when we consider the fixed prime polynomials $\mathcal{R}=T^9 + T^6 + T^4 + T^2 + 2T + 2$ and $\mathcal{S}=T^9 + T^6 + T^4 + T^3 + T^2 + T + 1$ in $\F_3[T]$. We have already seen that $\mathcal{R}$ is a c-Wieferich prime in $\F_3[T]$, but not $\mathcal{S}$. In this case, we have $E_{3,3}=\{0,1,2\}$. Since $s\equiv 0(\bmod\;3)$, we obtain c-Wieferich primes in $\F_{3^{1+3k}}[T]$ by first solving $X^3-X-k{\rm Tr}_{\F_{3^3}/\F_3}(\alpha)=0$. For $r=4$, solving $X^3-X-{\rm Tr}_{\F_{3^3}/\F_3}(\alpha)=0$ over $\F_{3^3}$ yields $X=z\in \F_{3^3}$ where $z^3+2z+1=0$. So $\mathcal{R}(T+z)=T^9 + T^6 + T^4 + T^3 + T^2 + T + 1=\mathcal{S}$, which is a c-Wieferich prime in $\F_{3^4}[T]$, that is defined over $\F_3$. Furthermore, for $r=7$, solving $X^3-X-2{\rm Tr}_{\F_{3^3}/\F_3}(\alpha)=0$ over $\F_{3^3}$ yields $x=2z\in \F_{3^3}$, where $z^3+2z+1=0$. We obtain the prime $\mathcal{T}=\mathcal{R}(T+2z)=T^9 + T^6 + T^4 + 2T^3 + T^2 + 2$ which is another fixed prime in $\F_3[T]$ that is a c-Wieferich prime in $\F_{3^7}[T]$. The other fixed primes in $\F_3[T]$ of degree $9$ are $T^9 + 2T^6 + 2T^4 + 2T^2 + 2T + 1, T^9 + 2T^6 + 2T^4 + T^3 + 2T^2 + T + 2$, and $T^9 + 2T^6 + 2T^4 + 2T^3 + 2T^2 + 1$. These are not c-Wieferich primes in any $\F_{3^k}[T]$. In fact, in $\F_{3^{10}}[T]$, the prime $\mathcal{R}$ resurfaces as a c-Wieferich prime, then $\mathcal{S}$ in $\F_{3^{13}}[T]$ and $\mathcal{T}$ in $\F_{3^{16}}[T]$ and the process/cycle is repeated.

\begin{cor}\label{fixedexists1}
Let $\mathcal{G}$ be a fixed subgroup of $\F_q$ of order $p$. If there exists a $\mathcal{G}$-fixed c-Wieferich prime in $\F_q[T]$ of degree $p$, then each $\mathcal{G}$-fixed prime in $\F_q[T]$ of degree $p$ is a c-Wieferich prime in at-least one of the constant field extensions $\F_{q^{r}}[T]$ of $\F_q[T]$ where $r\in \{1,\ldots,p-1\}$.
\end{cor}
\begin{proof}
Theorem \ref{impo3} shows that if there exists a $\mathcal{G}$-fixed c-Wieferich prime in $\F_q[T]$ of degree $p$, then we will find $p^{-1}q$ $\mathcal{G}$-fixed c-Wieferich primes in $\F_{q^r}[T]$ for $r=1+sk$ with $0\leqslant k< p-1$. All these primes are distinct giving a total $p^{-1}q(p-1)$ primes. Let $\mathcal{G}=\chi\F_p$ for some $\chi\in \F_q^{\times}$. The mapping $\theta:\F_q\to\F_q$, defined by $\theta(x)=x^p-\chi x$ is an $\F_p$-linear mapping with  has an image of size $p^{-1}q$. By Theorem \ref{CII?}, the number of almost Artin-Schreier primes in $\F_q[T]$ of degree $p$ is $q-p^{-1}q=p^{-1}q(p-1)$. By Theorem \ref{LOKI}, a $\mathcal{G}$-fixed prime of degree $p$ is an almost Artin-Schreier prime, and the result follows.
\end{proof}
\begin{rem}
Corollary \ref{fixedexists1} does not contradict the existence of more than one almost Artin-Schreier c-Wieferich primes in $\F_q[T]$. It asserts that when any such prime is considered, we can use Theorem \ref{impo3}. to find extensions where all the rest of the $\mathcal{G}$-fixed primes in $\F_q[T]$ will be c-Wieferich primes. For example, for $q=13$, there exists two distinct c-Wieferich primes $T^{13}+12T+1$ and $T^{13}+12+8$ of Artin-Schreier type. These stem from different $\alpha$s, in particular $\alpha\in B_{13,1}=\{5, 12\}$. But they will yield the same Artin-Schreier/c-Wieferich primes $\{T^{13}+12T+1,T^{13}+12T+2,\ldots,T^{13}+12T+12\}$ although in different extension rings. For example, if $q=13$ and $r=2$, then the ``line" of c-Wieferich primes $\mathcal{P}_1=T^{13}+12T+1$ yields $T^{13}+12T+7$ while that of $\mathcal{P}_2=T^{13}+12T+8$ yields $T^{13}+12T+4$. 
\end{rem}

From our extensive computations, we were unable to find any non-fixed c-Wieferich prime in $\F_q[T]$ for odd characteristic. However, for characteristic two fields, we found more examples of non-fixed c-Wieferich primes. These results generalize those of \cite[Section 5]{ASB2} on computing c-Wieferich primes. 

\section*{Acknowledgement}
We thank A. Keet, F. Breuer and D. Thakur for their patience and their helpful comments after reading the many drafts of this paper. The first author also thanks the second for hosting him as postdoc at Stockholm University. This research was (in part) carried out at the Department of Mathematics, Makerere University and at the Department of Mathematics of Stockholm University with financial support from the Makerere-Sida Bilateral Programme Phase IV, Project 316 ``Capacity building in Mathematics and its applications".



\nocite{*}
\bibliographystyle{amsplain}
\footnotesize{\bibliography{cwpARXIV}}
\end{document}